\documentclass[a4paper,12pt]{article}
\usepackage{a4wide}
\usepackage{amsfonts}
\usepackage{graphics}
\usepackage{amsmath, amsthm}  
\usepackage{amssymb,bbm} 
\usepackage{enumerate}

\usepackage{tikz}
\usepackage{epstopdf}

\newcommand{\id}{\mathbbm{1}}\newcommand{\N}{\mathbb{N}}
\newcommand{\Z}{\mathbb{Z}}\newcommand{\F}{\mathbb{F}}

\newcommand{\Lf}{\mathfrak{f}}

\newcommand{\Lm}{\mathfrak{m}}

\newcommand{\map}[3]{ #1 : #2 \longrightarrow #3 }

\newcommand{\mapl}[5]{ #1 : #2 \longrightarrow #3 : #4 \longmapsto #5 }

\newtheorem{theorem}{Theorem}[section]

\newtheorem{proposition}[theorem]{Proposition}
\newtheorem*{fakeproposition}{Proposition}
\newtheorem{corollary}[theorem]{Corollary}
\newtheorem*{corollary*}{Corollary}
\newtheorem*{fakecorollary}{Corollary}

\newtheorem{lemma}[theorem]{Lemma}

\newtheorem*{theorem*}{Theorem}
\newtheorem*{faketheorem}{Theorem}
\newtheorem*{proposition*}{Proposition}
\newtheorem*{lemma*}{Lemma}
\newtheorem*{problem*}{Problem}

\newtheorem*{observation*}{Observation}

\newtheorem{question}[theorem]{Question}

\theoremstyle{definition} \newtheorem{remark}[theorem]{Remark}
\theoremstyle{definition} \newtheorem{example}[theorem]{Example}
\theoremstyle{definition}\newtheorem{definition}[theorem]{Definition}

\title{Arithmetically-free group-gradings of Lie algebras: II}

\author{Wolfgang Alexander Moens\thanks{This research was supported by the Erwin Schr\"odinger Junior Fellowship (Grant XXXX), and the Austrian Science Foundation FWF (Grant J3371-N25).}}
\date{\today}

\begin{document}

\maketitle

\abstract{We study Lie algebras $L$ that are graded by an arbitrary group $(G,\ast)$ and have finite support, $X$. We show that $L$ is nilpotent of $|X|$-bounded class if $X$ is arithmetically-free. Conversely: if a finite subset $Y$ of $G$ is \emph{not} arithmetically-free, then $Y$ supports the grading of a non-nilpotent Lie algebra.}

\tableofcontents

\section{Introduction}

In this paper we consider gradings of Lie algebras $L$ by arbitrary groups $(G,\ast)$. We recall that such a grading is a decomposition $L := \bigoplus_{g \in G} L_g$ of $L$ into homogeneous subspaces, such that for all $g,h \in G$, we have $$ [L_g,L_h] \subseteq L_{g \ast h} .$$ A typical example of such a grading is the eigenspace decomposition of $L$ with respect to an automorphism or a derivation. Classical results in the literature conclude that, under reasonable conditions on the support $X := \{ g \in G | L_g \neq \{ 0 \} \}$ of the grading, the algebra $L$ must be nilpotent (of bounded class). We refer to the bibliography for examples of, and partial answers to, the following questions:

\begin{question} \label{QuestionArithmeticallyFree} Which properties of the support, $X$, guarantee that $L$ is nilpotent, or nilpotent of $|X|$-bounded class? \end{question}

\begin{question} \label{QuestionNotArithmeticallyFree} Which properties of a set $Y$ guarantee that it supports some non-nilpotent Lie algebra? \end{question}

In order to address these questions, we have introduced \emph{arithmetically-free} subsets of groups,  \cite{MoensPreprint}. A subset $X$ of an abelian group $(G,+)$ is arithmetically-free, iff $X$ is finite and $x,x + y,x + 2y,x + 3y,\ldots \subseteq X$ implies $y \not \in X$. More generally: a subset $X$ of an arbitrary group is arithmetically-free, iff $X$ is finite and every subset of $X$ of pairwise commuting elements is arithmetically-free. Perhaps the most obvious examples of arithmetically-free subsets of groups are subsets of $(\Z_p,+) \setminus \{ \overline{0} \}$ and finite subsets of $(\Z^m,+) \setminus \{ 0 \}$. \newline

If we restrict our attention to gradings with finite support, then the answers to \emph{Question} \ref{QuestionArithmeticallyFree} are given by theorem \ref{TheoremArithmeticallyFreeSupport}. We shall construct a \emph{generalised Higman map} $\map{H}{\N}{\N}$, such that:

\begin{faketheorem}[\ref{TheoremArithmeticallyFreeSupport}] Consider a Lie algebra $L$ that is graded by a group $G$. If the support $X$ of the grading is arithmetically-free, then $L$ is nilpotent of class at most $H(|X|)$.\end{faketheorem}

We offer some context. If the Lie algebra is spanned by finitely-many elements, then the nilpotency follows from Jacobson's theorem about weakly-closed sets of nilpotent operators on finite-dimensional vector spaces, \cite{Jacobson,JacobsonWeaklyClosed}. If the Lie algebra is generated by finitely-many elements, then the nilpotency is a consequence of Zel'manov's theorem for Lie algebras satisfying a polynomial identity in the context of the restricted Burnside problem, \cite{Efim}. We refer to remark \ref{RemarkEfim} for more details. \newline

Let us consider another special case -- a result that extends Thompson's solution of the Frobenius conjecture. In \cite{Higman}, Higman proved the existence of a map $\map{h}{\mathbb{P}}{\N}$ such that: 

\begin{theorem}[Higman] Consider a Lie ring $L$ that is graded by the simple group $(\Z_p,+)$. If the support does not contain $\overline{0}$, then $L$ is nilpotent of class at most $h(p)$. \end{theorem}

We note that Higman used a combinatorial trick to prove the existence of the map $\map{h}{\mathbb{P}}{\N}$, without providing an upper bound for $h(p)$. On the other hand: the generalised Higman map $\map{H}{\N}{\N}$ will be the solution to an explicit recursion.  It does indeed grow very quickly, but if we assume that the support of the grading also admits a good-ordering (in the sense of Shalev, \cite{Shalev}), then we obtain a stronger upper bound. In \cite{MoensPreprint}, we prove:

\begin{theorem} Consider a Lie algebra $L$ that is graded by a group $G$. If the support $X$ is arithmetically-free and admits a good-ordering, then $L$ is nilpotent of class at most $1 + |X| + |X|^2 + |X|^3 + \cdots + |X|^{2^{|X|}}.$\end{theorem}

In particular: we recover results by Kreknin-Kostrikin and Khukhro, \cite{KrekninKostrikin,KhukhroSupport}. We refer to \cite{MoensPreprint} for a more detailed exposition about arithmetically-free subsets of groups, and their applications to finite group theory and Lie theory. \newline

We will deduce theorem \ref{TheoremArithmeticallyFreeSupport} from the following characterisation of arithmetically-free subsets of groups.

\begin{faketheorem}[\ref{TheoremCharacterisation}] Consider a finite subset $X$ of an abelian group $(G,+)$, and let $<$ be a total order on $G$. Then the following two properties are equivalent:
\begin{enumerate}
\item The set $X$ is arithmetically-free.
\item Every sequence $S = (g_1,g_2,\ldots,g_{H(|X|) + 1})$ on $G$ of length $H(|X|) + 1$ has an initial segment $(g_1,g_2,\ldots,g_k) \text{ with } g_1 + g_2 + \cdots + g_k \in G \setminus X,$ or a Lie-regular segment $(g_i,g_{i+1},\ldots,g_j) \text{ with } g_i + g_{i+1} + \cdots + g_j \in G \setminus X.$
\end{enumerate}
\end{faketheorem}


Lie-regularity is a minimality condition under the action of permutations derived from the Jacobi-identity; it will be introduced in \emph{Definition $3$}. We will prove the theorem by generalising the definitions and techniques of \cite{Higman}. \newline

We also answer Question \ref{QuestionNotArithmeticallyFree}:

\begin{fakeproposition}[\ref{PropositionMetabelian}] If a finite subset $Y$ of a group $G$ is \emph{not} arithmetically-free, then there exists a non-nilpotent Lie algebra $L$ that is graded by $G$ such that the support is $Y$. \end{fakeproposition}

Such a Lie algebra can be made to satisfy $\dim(L) = |Y|$ and $\operatorname{dl}(L) = 2$. But there is a special Lie algebra, the standard filiform Lie algebra over a field $\F$, which characterises \emph{all} (non)-arithmetically-free sets. This algebra is defined by the presentation, $\Lf := \langle v , w_1 , w_2 , \ldots | [v,w_i] = w_{i+1} , [w_j,w_k] = 0 \rangle,$ where $i,j$ and $k$ run over all natural numbers.

\begin{fakecorollary}[\ref{PropositionFiliform}] A finite subset $X$ of an arbitrary group is arithmetically-free, iff $X$ does not support a grading of the standard filiform Lie algebra. \end{fakecorollary}

\paragraph{Structure of the text.} In section \ref{SectionArithmeticallyFreeSubsets}, we introduce arithmetically-free subsets of groups. In section \ref{SectionGroupGradings}, we prove the main results mentioned in the introduction modulo theorem \ref{TheoremCharacterisation}. In section \ref{SectionCharacterisation}, prove theorem \ref{TheoremCharacterisation}, which characterises arithmetically-free sets by means of Lie-regularity. We conclude with a remark about walks in Cayley-graphs.

\section{Arithmetically-free subsets of groups} \label{SectionArithmeticallyFreeSubsets}

Let us begin by introducing some notation. For an element $g$ of a group $(G,+)$, we let $\langle g \rangle$ be the cyclic subgroup of $G$ that is generated by $g$. For elements $g,x \in G$ and $n \in \N$, we define the (partial) orbit, or (partial) arithmetic progression $$\mathcal{O}_n(x,g) := \{ x , x + g , x + 2 \cdot g , \ldots , x + n \cdot g \},$$ with base $x$ and increment $g$. The (full) orbit, or (full) arithmetic progression corresponding with $x$ and $g$ is $\mathcal{O}(x,g) := \{ x , x + g , x + 2 \cdot g , x + 3 \cdot g , \ldots \} = \bigcup_n \mathcal{O}_n(x,g).$

\begin{definition} [Arithmetically-free subsets] Consider an abelian group $(G,+)$ and a subset $X$ of $G$. We say that $X$ is arithmeticically-free in $G$, iff $X$ is finite and satisfies either (both) of the following properties:
\begin{enumerate}
\item If $x,g \in X$, then $\mathcal{O}(x,g) \not \subseteq X,$
\item If $x,g \in X$, then $\mathcal{O}_{|X|}(x,g) \not \subseteq X.$
\end{enumerate}
More generally: a subset $\widetilde{X}$ of an arbitrary group $(\widetilde{G},\ast)$ is arithmetically-free, iff $\widetilde{X}$ is finite, and every subset $X$ of pairwise commuting elements is arithmetically-free.
\end{definition}

\begin{example} Let $(G,\ast)$ be an arbitrary group and let $X$ be a finite subset. If $\operatorname{ord}_{G}(x) > |X|$ for every $x \in X$, then $X$ is arithmetically-free. In particular:
\begin{enumerate}
\item Let $(\Z_n,+)$ be the cyclic group of order $n \in \N$. Then the subset $X := \{ x \in \Z_n | \operatorname{ord}_{(\Z_n,+)}(x) = n \}$ is arithmetically-free in $(\Z_n,+)$, and so are all subsets of $X$.
\item Let $(G,+)$ be the free-abelian group $(\Z^m,+)$ of rank $m \in \N$. Then the arithmetically-free subsets of $G$ are precisely the \emph{finite} subsets of $G \setminus \{ 0 \}$.
\end{enumerate}
\end{example}

\begin{example} Let $(G,\ast)$ be an arbitrary group. If the finite subset $X$ of $G$ is sum-free, then $X$ is arithmetically-free. \end{example}

We refer to \cite{MoensPreprint} for more examples and their applications.

\section{Group-gradings of Lie algebras} \label{SectionGroupGradings}

In this section we will prove theorem \ref{TheoremArithmeticallyFreeSupport} and propositions \ref{PropositionMetabelian} and \ref{PropositionFiliform} modulo theorem \ref{TheoremCharacterisation}. We begin by performing some elementary reductions.

\subsection{Reduction to finite abelian grading groups}

Suppose we are given a grading $\bigoplus_{g \in G} L_g$ of a Lie algebra $L$ by a group $G$, with support $X$. If $H$ is a subgroup of $G$, then we may define the homogeneous subalgebra $$L_H := \bigoplus_{h \in H} L_h$$ of $L$, which is clearly graded by $H$, with support $X \cap H$. We note that if $L$ is nilpotent of class at most $c \in \N$, then for every abelian subgroup $H$ of $G$, the subalgebra $L_H$ is nilpotent of class at most $c$. The following lemma from \cite{MoensPreprint} shows that the converse is also true.

\begin{lemma} \label{LemmaAbelianSupport} Consider $c \in \N$, an \emph{arbitrary} group $G$, and a $G$-graded Lie algebra $L$ with support $X$. The algebra $L$ is nilpotent of class at most $c$, iff for every abelian subgroup $H$ of $G$, the homogeneous $H$-graded subalgebra $L_H$ of $L$ with support $X \cap H$ is nilpotent of class at most $c$. \end{lemma}

So, in order to prove the nilpotency of graded Lie algebras, it suffices to consider abelian grading groups.

\begin{definition}[Contraction] Consider two groups $(G,+)$ and $(H,+)$, with respective subsets $X$ and $Y$. A contraction of $X$ onto $Y$ is a surjective map $\map{f}{X}{Y}$ such that for all $x_1,x_2,x_3 \in X$ with $x_1 + x_2 = x_3$, we have $f(x_1) + f(x_2) = f(x_3).$ \end{definition}

\begin{example} Every homomorphism is a contraction. A subset $X \subseteq (G,+)$ is sum-free, iff $X$ contracts onto $\{ \overline{1} \}\subseteq (\Z_2,+)$. \end{example}

Contractions allow us to take a grading of a Lie algebra, and make the support smaller, or ``less complicated.''

\begin{proposition} \label{PropositionContraction} Consider a grading $\bigoplus_{g \in G} L_g$ of a Lie algebra $L$ by the group $(G,+)$ with support $X$. If $\map{f}{X}{Y}$ is a contraction of $X$ onto $Y \subseteq (H,+)$, then $$ L^y := \bigoplus_{x \in f^{-1}(y)} L_x \text{ for } y \in Y, \text{ and } L^h := \{ 0 \} \text{ for all } h \in H \setminus Y$$ defines a grading of $L$ by $(H,+)$ with support $Y$. \end{proposition}

\begin{proof} For arbitrary elements $h,h'$ of $H$, we have $$ [L^h,L^{h'}] \subseteq \bigoplus_{\substack{x \in f^{-1}(h)\\x' \in f^{-1}(h')}} [L_{x},L_{x'}]  \subseteq \bigoplus_{\substack{x \in f^{-1}(h)\\x' \in f^{-1}(h')}} L_{x + x'} \subseteq L^{h + h'}.\qedhere$$
\end{proof}

We note that much information may be lost in this construction. Indeed: the trivial homomorphism $\mapl{\varepsilon}{G}{H}{x}{0}$ shows that contractions need not map arithmetically-free sets  onto arithmetically-free sets. We do have the following result.

\begin{lemma} Consider an abelian group $(G,+)$ and an aritmetically-free subset $X$. Then there exists a homomorphism $\map{\pi}{G}{H}$ from $G$ to a finite abelian group $H$ such that $\pi(X)$ is arithmetically-free. \end{lemma}

\begin{proof} Let $G$ be given by $(\Z^m \oplus F,+)$, where $m \in \N$ and $F$ is a finite abelian group. Since $X$ is finite, there exists a natural number $R$ such that for all $x := (x_1,\ldots,x_m,x_{m+1}) \in X$: $0 < | x_1| , \ldots, |x_m| < R / 3$. Now define the natural projection $\map{\pi}{\Z^m \oplus F}{\Z_R^m \oplus F}$. If $\pi(x) + \pi(y) = \pi(z)$, for some $x,y,z \in X$, then $R | (x_1 - y_1 - z_1),\ldots,(x_m - y_m - z_m)$, so that $x = y + z$. In particular: if, for some $x,g \in X$ and $n \in \N$, we have $\mathcal{O}_n(\pi(x),\pi(g)) \in \pi(X)$, then we also have $\mathcal{O}_n(x,g) \in X$.\end{proof}

By combining proposition \ref{PropositionContraction} with lemma \ref{LemmaFiniteSupport}, we obtain:

\begin{corollary} \label{CorollaryFiniteAbelianSupport} Consider a grading $\bigoplus_g L_g$ of a Lie algebra $L$ by an abelian group $G$. Suppose that the support, $X$, is arithmetically-free. Then there exists a finite abelian group $H$ and a grading of $L$ by $H$ such that the support, $Y$, is arithmetically-free and $|Y| \leq |X|$.\end{corollary}

So, when proving the nilpotency of graded Lie algebras with arithmetically-free support, we may assume that the grading group is finite and abelian. We can now explain the connection between theorem \ref{TheoremArithmeticallyFreeSupport} and a result of Zel'manov.

\begin{remark} \label{RemarkEfim} Consider a graded Lie algebra $L$ with arithmetically-free support $X$. Let us show that $L$ is nilpotent if it is \emph{finitely-generated}. By lemma \ref{LemmaAbelianSupport}, we may suppose that the grading group is abelian, and by corollary \ref{CorollaryFiniteAbelianSupport}, we may even assume that it is finite abelian. Since $X$ is arithmetically-free, the component $L_0$ satisfies a non-trivial polynomial identity. A result of Bahturin-Zaicev lifts this identity to a polynomial identity on all of $L$, \cite{BahturinZaicev}. Since $L$ is finitely generated and $X$ is finite, $L$ is also generated by finitely many homogeneous elements. Since all homogeneous elements are $\operatorname{ad}$-nilpotent of index at most $|X|$, we may apply Zel'manov's theorem to conclude that $L$ is nilpotent, \cite{Efim}. \newline

The theorems of Bahturin-Zaicev and Zel'manov have been extended (by Shumyatsky) to guarantee the existence of upper bounds for the degree of the polynomial identity and the nilpotency class, \cite{Shumyatsky}. But, unfortunately, these bounds depend implicitly on the number of generators of $L$.\end{remark}

\subsection{Group-gradings with non-arithmetically-free support}

For each $n \in \N$, $n > 1$, we may define the standard metabelian Lie algebra $\Lm(n)$ over a field $\F$ by the presentation: $ \Lm(n) := \langle v,w_0,\ldots,w_{n-1} | [v,w_i] = w_{i+1} , [w_j,w_k] = 0 \rangle ,$ where $i,j,k$ run over the elements of $(\Z_n,+)$. Let us consider an obvious grading of $\Lm(n)$. For any $g \in G$ of order $n$, and $x \in \operatorname{Centr}_G(g) \setminus \langle g \rangle$, we define $\Lm(n)_g := \F \cdot v$ and $\Lm(n)_{x + i \cdot g} := \F \cdot w_i,$ with $i \in \Z_n.$ Then this extends to a $G$-grading of $\Lm(n)$ with support $\mathcal{O}(x,g)$.

\begin{proposition} \label{PropositionMetabelian} Consider a finite subset $Y$ of a group $(G,\ast)$, and suppose that $Y$ is \emph{not} arithmetically- free. Then there exists a non-nilpotent Lie algebra $L$ that is graded by $G$ with support $Y$. It is metabelian and its dimension is $|Y|$. \end{proposition}

\begin{proof} We may suppose that $Y$ does not contain the neutral element of $G$ (since otherwise every Lie algebra is supported by $Y$). Since $Y$ is not arithmetically-free, there is a subset $Y_0$ of $Y$ of pairwise commuting elements that is not arithmetically-free. So there exist $x,g \in Y_0$ with $\mathcal{O}(x,g) \subseteq Y_0 \subseteq Y$. Clearly, $g$ has finite order $|\mathcal{O}(x,g)| -1$, and $x \not \in \langle g \rangle$. So $\Lm(|\mathcal{O}(x,g)| -1)$ is graded by $G$ with support $\mathcal{O}(x,g) \subseteq Y$. The abelian Lie algebra $\F^{|Y|-|\mathcal{O}(x,g)|}$ is trivially graded by $G$ with support $Y \setminus \mathcal{O}(x,g)$. So Lie algebra $L := \F^{|Y| - |\mathcal{O}(x,g)|} \oplus \Lm(|\mathcal{O}(x,g)|-1)$ is $G$-graded with support $Y$. It has dimension $|Y|$ and satisfies $[[L,L],[L,L]] = 0$. \end{proof}

Alternatively, we may consider the (infinite) standard filiform Lie algebra $\Lf$ over a field $\F$, which is given by the presentation $\Lf := \langle v , w_1,w_2,\ldots | [v,w_i] = w_{i+1} , [w_j,w_k] = 0 \rangle,$ where $i,j,k$ run over $\N$. This algebra is metabelian (but not nilpotent), and it also has obvious gradings. Let $g,x \in G$ be as before: $g$ has order $n \in \N$ and $x \in G \setminus \langle g \rangle.$ We define $\Lf_g := \F \cdot v$, and for each $i \in \N$, we set $\Lf_{x + i \cdot g} := \bigoplus_{j \geq 0}\F \cdot w_{i + j \cdot n}$ Then this extends to a $G$-grading of $\Lf_{\infty}$ with support $\mathcal{O}(x,g)$. As before, we obtain:

\begin{proposition} \label{PropositionFiliform} Consider a finite subset $Y$ of a group $G$, and suppose that $Y$ is \emph{not} arithmetically-free. Then the standard filiform Lie algebra $\Lf$ has a $G$-grading with support $Y_0 \subseteq Y$. \end{proposition}

\subsection{Lie-regularity of homogeneous words}

In this subsection we introduce Lie-regular words, which turn out to be very useful in proving the nilpotency of graded algebras. We first recall some standard notation. If $L$ is a Lie algebra, and $v_1,\ldots,v_k,v_{k+1} \in L$, then we recursively define the left-associative products by $[v_1,\ldots,v_k,v_{k+1}] := [[v_1,\ldots,v_k],v_{k+1}]$. The bilinearity of the Lie-bracket together with the Jacobi-identity guarantee that every product of $k$ elements can be written as a linear combination of left-associative products of length $k$. Let us make this precise.

\paragraph{Linearisations and Lie-regularity.} Consider the free Lie algebra $F$ with free generating set $u_{1},\ldots,u_{k},$ $v_{1},\ldots,v_{l},$ $w_{1},\ldots,w_{m}.$ Let $P$ be an arbitrary (not necessarily left-associative) product of the elements $v_1,\ldots,v_l$, e.g.: $P = [[v_1,v_2],[v_3,v_4]]$ or $P = [v_1,[[v_2,v_3],v_4]]$. By applying the Jacobi-identity at most $(l-1)$-times, we obtain a set of permutations $\mathcal{S}_P \subseteq \operatorname{Sym}(l)$ containing the identity, a map $\map{\varepsilon_{P}}{\mathcal{S}_P}{\{ \pm 1\}}$, and the \emph{linearisation of $P$}:
$$
[ u_{1},\ldots,u_{k} , P , w_{1},\ldots,w_{m}] = \sum_{\pi \in \mathcal{S}_P} \varepsilon_{P}(\pi) \cdot [u_{1},\ldots,u_{k},v_{{\pi(1)}},\ldots,v_{{\pi(l)}},w_{1},\ldots,w_{m}].
$$

\begin{example} Let $l := 4$, and $P := [[v_1,v_2],[v_3,v_4]]$. The linearisation of $P$ is then given by the $8$ permutations $\id,(12),(34),(12)(34),(13)(24),(1324),(1423),(14)(23)$ and their respective signs $+,-,-,+,-,+,+,-$. \end{example}

\begin{definition}[Lie-regularity] \label{DefinitionLieRegularity} Let $(A,<)$ be a totally-ordered set, and give $A^l$ the lexicographical order. The sequence $(a_1,\ldots,a_l) \in A^l$ is \emph{Lie-regular}, iff there exists a set of permutations $\mathcal{S}_P$ defining the linearisation of a Lie product $P$, such that $(a_{1},\ldots,a_{l}) < (a_{\pi(1)},\ldots,a_{\pi(l)})$, for all $\pi \in \mathcal{S}_P \setminus \{ \id \}$. \end{definition}

We briefly mention that it also makes sense to consider linearisations in Lie algebras that satisfy various identities, say: metabelian or solvable Lie algebras. The corresponding notion of regularity will then yield stronger conclusions (i.e.: explicit upper bounds for the nilpotency class). We refer to \cite{MoensPreprint} for examples of such bounds.

\paragraph{Graded linearisation.} The above linearisation formula holds in every Lie algebra. But, if the Lie algebra is graded, then the linearisations become more interesting. So let us consider a grading $\bigoplus_g L_g$ of a Lie algebra $L$ by a finite abelian group $(G,+)$ and let $X \subseteq G$ be its support. Let $<$ be a total order on $X$, and extend it lexicographically to all finite $X$-tuples. \newline

We define the homogenous, left-associative words (of the grading) to be the homogeneous elements $v_1,v_2,\ldots$ themselves, and their iterated products $[v_1,v_2], [v_1,v_2,v_3], \ldots$. To every homogeneous element $v$ that is not zero, we may associate a unique element $\overline{v} \in X$ such that $v \in L_{\overline{v}}$. If the elements $v_1,\ldots,v_k$ are homogeneous and not zero, then we assign to the word $[v_1,v_2,\ldots,v_k]$ the weight sequence $(\overline{v}_1,\ldots,\overline{v}_k) \in X^k \subseteq G^k$. If $[v_1,v_2,\ldots,v_k]$ and $[w_1,w_2,\ldots,w_l]$ are two such words, we may compare them: $$[v_1,v_2,\ldots,v_k] < [w_1,v_2,\ldots,w_l] \leftrightarrow (\overline{v}_1,\ldots,\overline{v}_k) < (\overline{w}_1,\ldots,\overline{w}_l).$$

Let us evaluate the identity $(\ast)$ in homogeneous elements of the grading $\bigoplus_g L_g$. Suppose that $v_1,\ldots,v_k \neq 0$, but $\overline{v_1} + \cdots + \overline{v_l} \in G \setminus X$. Then $P = 0$, so that the identity collapses to
$$
0 = \sum_{\pi \in \mathcal{S}_P} \varepsilon_{P}(\pi) \cdot [u_{1},\ldots,u_{k},v_{{\pi(1)}},\ldots,v_{{\pi(l)}},w_{1},\ldots,w_{m}].
$$

We have (trivially) shown:

\begin{proposition} \label{PropositionIncrease} Consider a non-zero, left-associative, homogeneous word $v$ with weight sequence $(a_1,\ldots,a_k,b_1,\ldots,b_l,c_1,\ldots c_m) \in X^{k + l + m}$. If the segment $(b_1,\ldots,b_l)$ is Lie-regular and $ b_1 + \cdots + b_l \in G \setminus X$, then $v$ is spanned by strictly larger, left-associative, homogeneous words of the same length. \end{proposition}

Similarly, we obtain:

\begin{proposition} \label{PropositionInitial} Consider non-zero, homogeneous elements $v_1,\ldots,v_k,w_1,\ldots,w_l$. If $\overline{v}_1 + \cdots + \overline{v}_k \in G\setminus X$, then $[v_1,\ldots,v_k,w_1,\ldots,w_l] = 0.$ \end{proposition}

\subsection{Lie algebras with arithmetically-free support}

We are now in a position to prove theorem \ref{TheoremArithmeticallyFreeSupport} modulo theorem \ref{TheoremCharacterisation}. 

\begin{theorem} \label{TheoremArithmeticallyFreeSupport} Consider a Lie algebra $L$ that is graded by a group $G$. If the support, $X$, of the grading is arithmetically-free, then $L$ is nilpotent of class at most $H(|X|)$. \end{theorem}

\begin{proof} Let the grading be given by $L = \bigoplus_{g \in G} L_g$. In view of lemma \ref{LemmaAbelianSupport}, we may assume that $G$ is abelian. Let $\mathcal{W}$ be the set of non-zero, left-associative, homogeneous words (w.r.t this grading) of length $H(|X|) + 1$. Let $(X,<)$ be an arbitrary total order on $X$, and order all elements of $\mathcal{W}$ by their weight sequences. Suppose $\mathcal{W} \neq \emptyset$. Since there are only finitely many sequences on $X$ of length $H(|X|) +1$, we may assume that $W \in \mathcal{W}$ is maximal. Let $S$ be the weight sequence of $W$. If $S$ has an initial segment with weight in $G \setminus X$, then $W$ vanishes according to proposition \ref{PropositionInitial}. This contradicts the choice of $W$. If $S$ has a Lie-regular segment of weight in $G \setminus X$, then according to proposition \ref{PropositionIncrease}, $W$ is a linear combination of $W' \in \mathcal{W}$, with $W' > W$. Since $W$ was chosen maximal-non-zero, all of the $W'$ vanish. This again contradicts the choice of $W$. But theorem \ref{TheoremCharacterisation} guarantees that one of these cases must occur. We conclude $\mathcal{W} = \emptyset$, so that $L$ is nilpotent of class at most $H(|X|)$. \end{proof}

We recall some notation: the lower central series of a Lie algebra $L$ is defined by $\gamma_1(L) := [L,L]$, and $\gamma_{t+1} := [\gamma_t(L),L]$ for $t \in \N$, $t > 1$.

\begin{corollary} Consider a Lie algebra $L$ that is graded by a group $G$. Let $X$ be an arithmetically-free subset of $G$ and let $I$ be a homogeneous ideal of $L$ containing the homogeneous subspace $\bigoplus_{g \in G \setminus X} L_g$ of $L$. Then $\gamma_{H(|X|)+1}(L) \subseteq I$. \end{corollary}

\section{A characterisation of arithmetically-free sets} \label{SectionCharacterisation}

\begin{definition}[Alphabet, weight] An \emph{alphabet} with coefficients in an abelian group $(G,+)$ is a totally ordered set $(A,<)$ together with a map $\map{\omega}{A}{G}$, called the \emph{weight}, and an arithmetically-free subset $X$ of $G$. Notation: $(A,<,\map{\omega}{A}{G},X)$, or simply $(A,<,\omega,X)$. We say that the alphabet is \emph{finite}, if $\omega^{-1}(X)$ is a finite subset of $A$. \end{definition}

In \ref{DefinitionGeneralisedHigmanMap}, we recursively define a map $\mapl{H}{\N \times \N \times \N}{\N}{(a,b,c)}{H_a(b,c)}$. The main aim of this section is to prove the following theorem.

\begin{theorem} \label{TheoremAlphabet} Consider a finite alphabet $(A,<,\map{\omega}{A}{G},X)$, and let $S := (a_1,\ldots,a_{k})$ be a sequence on $A$ of length $k \geq {H}_{|X|}(|X|,|\omega^{-1}(X)|)$. Then $\{ 0, \omega(a_1), \omega(a_1) + \omega(a_2) , \ldots , \omega(a_1) + \cdots \omega(a_k) \} \subseteq G$ cannot be translated into $X$, or $S$ contains a Lie-regular segment $(a_i,\ldots,a_j)$ with $\omega(a_i) + \cdots + \omega(a_j) \in G \setminus X$. \end{theorem}

We recall that a subset $M$ of $G$ can be translated into a subset $N$ of $G$, iff there exists an $h \in G$ such that $h + M := \{ h + m | m \in M\} \subseteq N$. Set $H(n) := H_n(n,n)$. If we specialise $A := G$ and $\mapl{\omega}{A}{G}{a}{a}$, then we obtain:

\begin{theorem}[Characterisation of arithmetically-free subsets of groups] \label{TheoremCharacterisation} Consider an abelian group $(G,+)$ with total order $<$ and a finite subset $X$. The following two properties are then equivalent:
\begin{enumerate}
\item $X$ is arithmetically-free.
\item Every sequence $S := (g_1,g_2,g_3, \ldots , g_{k})$ on $G$ of length $k \geq H(|X|)+1$, has an initial segment $(g_1,g_2,\ldots,g_{l})$ satisfying $g_1 + g_2 + \cdots + g_l \in G \setminus X$, or has a Lie-regular segment $(g_i,g_{i+1},\ldots,g_j)$ satisfying $g_i + g_{i+1} + \cdots + g_j \in G \setminus X$. 
\end{enumerate} \end{theorem}

\begin{proof} Suppose first that $X$ is arithmetically-free, and let $S := (g_1,\ldots,g_k) \in G^k$, with $k \geq H(|X|) + 1$. Define the $T := (g_2,\ldots,g_k) \in G^{H(|X|)}$. If $\{ 0,g_2,g_2+g_3,\ldots,g_2 + \cdots + g_k \}$ cannot be translated into $X$, then neither can $g_1 + \{ 0,g_2, g_2 + g_3 , \ldots, g_1 + \cdots + g_k \} = \{ g_1,g_1 + g_2 , \ldots , g_1 + \cdots + g_k \}$. Else, theorem \ref{TheoremAlphabet} gives us a Lie-regular segment $U := (g_i, \ldots,g_j)$ of $T$ (and therefore of $S$) satisfying $g_i + \cdots + g_j \in G \setminus X$. \newline

Conversely, suppose that $X$ is not arithmetically-free. Then there exist $x,g \in X$ with $\mathcal{O}(x,g) \subseteq X$. For each $k \in \N$, we define $S_k := (x,g,\ldots,g) \in X^{k+1}$. It is now easy to verify that the $S_1,S_2,\ldots$ form a family of arbitrarily long sequences on $G$ for which every initial and every Lie-regular segment $T :=(a_1,\ldots,a_l)$ satisfies $a_1 + \cdots + a_l \in X$. Indeed: the initial segments of $S_k$ are $S_1,\ldots,S_k$. If $x < g$, then the Lie-regular segments are again precisely $S_1,S_2,\ldots,S_k$. Else, $S_k$ has no Lie-regular segments at all. \end{proof}

This was the theorem used to prove theorem \ref{TheoremArithmeticallyFreeSupport}. If we specialise even more ($G := \Z_p$, for some prime $p$, and $X := \Z_p \setminus \{ 0 \}$), then we recover the central result (lemma $3$) of \cite{Higman}. In particular: $h(p) \leq H(p-1)$, for all primes $p$. \newline

Before sketching the proof of theorem \ref{TheoremAlphabet}, we present some useful invariants.

\begin{definition}[Weight, Span, Content] Consider a finite alphabet $(A,<,\omega,X)$, and let $S := (a_1,\ldots,a_k)$ be a sequence on $A$. The \emph{weight} of $S$ is $\omega_1(S) := \omega(a_1) + \cdots + \omega(a_k)$. The \emph{span} of $S$ is $\sigma(S) := \{ 0 , \omega(a_1), \omega(a_1) + \omega(a_2), \ldots, \omega(a_1) + \cdots + \omega(a_k) \} \subseteq G.$ We also define $\alpha(S) := \omega^{-1}(X) \cap \{ a_1,\ldots,a_k \} \subseteq A$. The \emph{content} of $S$ is $\kappa(S) := (|\sigma(S)|,|\alpha(S)|) \in \N \times \N.$ \end{definition}

Let us \emph{informally} define a sequence $S := (a_1,\ldots,a_k) \in A^k$ on $A$ to be \emph{full}, if $S$ satisfies the conclusions of theorem \ref{TheoremAlphabet} in a particularly strong sense. (For a more precise definition, we refer to \ref{DefinitionFullness} and \ref{LemmaRegularityImpliesLieRegularity}).

\paragraph{Strategy of the proof.} This will be a (relatively simple) induction on the type of the sequence. Let us give $\N \times \N$ the lexicographical order $<_{\operatorname{Lex}}$, and let us fix a finite alphabet $\mathcal{A} := (A,<,\omega,X)$. Let us first show that a finite sequence $S$ on $A$ satisfying $\| S \| \geq H_{|X|}(\kappa(S))$ is full. If $\kappa(S) \in \N \times \{ 1 \}$ or $\{ 1 \} \times \N$, or if $|\sigma(S)| > |X|$, then $S$ will be full for trivial reasons. In the other case, we construct a new finite alphabet $\mathcal{A}' := (A',<',\map{\omega'}{A'}{G},X)$, and a new sequence $S'$ on $A'$ such that:
\begin{enumerate}
\item If $S'$ is full (w.r.t. $\mathcal{A}'$), then $S$ is full (w.r.t. $\mathcal{A}$),
\item $\kappa(S') <_{\operatorname{Lex}} \kappa(S)$, and
\item $\| S' \| \geq H_{|X|}(\kappa(S'))$.
\end{enumerate}
The induction step then yields the above claim. By construction, the map $\map{H_{|X|}}{\N \times \N}{\N}$ is monotone increasing in both arguments. Now suppose that $\| S \| \geq H_{|X|}(|X|,|\omega^{-1}(X)|)$. If $|\sigma(S)| \leq |X|$, then $S$ is full by the above. Else, $|\sigma(S)| > |X|$, so that $S$ is trivially full. This finishes the proof.

\subsection{Restrictions and shifts of alphabets}

\begin{definition}[Restriction] Consider a finite alphabet $\mathcal{A} := (A,<,\map{\omega}{A}{G},X)$ and a subset $A'$ of $A$. Let $<'$ be the restriction of $<$ to $A'$, and let $\map{\omega'}{A'}{G}$ be the restriction of $\omega$ to $A'$. Then we say that the finite alphabet $(A',<',\map{\omega'}{A'}{G},X)$ is the \emph{restriction} of $\mathcal{A}$ to $A'$. \end{definition}

If $A$ is a set, we define $W^0(A) := A$, and we let $W(A) : W^1(A)$ be the set of finite sequences on $A$. Inductively, we then define $W^{n+1}(A) := W(W^n(A))$, for all $n \in \N$. We also define $$W^\ast(A) := \bigcup_{n \geq 0} W^n(A).$$

The length $\| S \|$ of an element $S := (x_1,\ldots,x_k) \in W^{n+1}(A)$, with $x_1,\ldots,x_k \in W^n(A)$, is defined to be $k$.

\begin{definition}[Shift] Consider a finite alphabet $\mathcal{A} := (A,<,\map{\omega}{A}{G},X)$. Let $<_1$ be the lexicographical order on $W(A)$, and define $\mapl{\omega_1}{W(A)}{G}{(x_1,\ldots,x_k)}{\omega(x_1) + \cdots + \omega(x_k)}.$ Then $(W(A),<_1,\omega_1,X)$ is again an alphabet, the \emph{shift} of $\mathcal{A}$. By iteration, we obtain a family of alphabets $(W^n(A),<_n,\omega_n,X)_{n \in \N}$. For $S := (x_1,\ldots,x_k) \in W(W^n(A))$, we define the \emph{span} to be $ \sigma(S) := \omega_n(x_1) + \cdots + \omega_n(x_k).$ \end{definition}

\subsection{Regularity and Lie-regularity}

We recall that Lie-regularity was defined by means of identities (linearisations) in the free Lie algebra. It is possible, and it will be convenient, to define a family Lie-regular elements by combinatorial means, and without having to invoke permutations. This will be the family of (simply) regular elements with respect to a fixed alphabet.

\begin{definition}[Regularity and type] Consider an alphabet $\mathcal{A} := (A,<,\omega,X)$. All elements of $W^0(A) = A$ are regular and of the same type. Now let $n \in \N \cup \{ 0 \}$. An element $(x_1,\ldots,x_k)$ of $W^{n+1}(A)$, with $x_1,\ldots,x_k \in W^n(A)$, is regular, iff
\begin{enumerate}
\item The elements $x_1,\ldots,x_k$ are all regular, and
\item The elements $x_1,\ldots,x_k$ are all of the same type, and
\item $\exists 1 \leq l < k$ such that $x_1 = x_2 = \cdots = x_l <_n x_{l+1},\ldots,x_k$.
\end{enumerate}
Two regular elements $(x_1,\ldots,x_k)$ and $(y_1,\ldots,y_l)$ of $W^{n+1}(A)$ are of the same type, iff $x_1 = y_1$. 
\end{definition}

\begin{definition}[Underlying sequence] Let $(A,<)$ be a totally-ordered set. We define the family $(\map{\pi_n}{W^{n+1}(A)}{W^{n}(A)})_{n \geq 0}$ of maps by $$ \pi_{n}(((x_1,\ldots,x_k),\ldots,(z_1,\ldots,z_m))) := (x_1,\ldots,x_k,\ldots,z_1,\ldots,z_m) \in W^{n}(A) .$$ We further define the map $\map{\pi}{W^\ast(A)}{W(A)}$, by: for all $S \in W^n(A)$, 
\begin{eqnarray*}
\pi(S) &:=& \pi_1 \circ \cdots \circ \pi_{n-1} (S) \in W(A) \text{ if } n \geq 2, \\
\pi(S) &:=& S \text{ if } n = 1, \\
\pi(S) &:=& (S) \text{ if } n = 0.
\end{eqnarray*}
The \emph{underlying sequence} of an element $S \in W^{\ast}(A)$ is defined to be $ \pi (S) \in W(A) .$ \end{definition}

\begin{lemma}[Regularity implies Lie-regularity] \label{LemmaRegularityImpliesLieRegularity} Consider an alphabet $\mathcal{A} := (A,<,\omega,X)$. If an element $S \in W^{n+1}(A)$ is regular, then the underlying sequence of $S$ is Lie-regular. \end{lemma}

\begin{proof} We refer to lemma $6$ of \cite{Higman} for the proof. It inducts on $n \geq 0$, and it implicitly constructs the set $\mathcal{S}_P \subseteq \operatorname{Sym}(r)$ and the map $\map{\varepsilon_P}{\mathcal{S}_P}{\{ \pm 1\}}$. \end{proof}

This observation, together with theorem \ref{TheoremCharacterisation}, inspires us to define fullness.

\begin{definition}[Fullness] \label{DefinitionFullness} Consider an alphabet $\mathcal{A} := (A,<,\omega,X)$. A sequence $S := (a_1,\ldots,a_k) \in A^k$ on $A$ is \emph{full}, iff: $(1)$ $\sigma(S)$ cannot be translated into $X$, or $(2)$ some segment $T$ of $S$ is the underlying sequence of a regular element of $W^\ast(A)$ with $\omega_1(T) \in G \setminus X$. \end{definition}

\begin{example} Consider an alphabet $\mathcal{A} := (A,<,\omega,X)$ and the elements $a < b < c$ in $A$. We define the sequences $R := (a,b,c), S := (a,a,b), T := (a,b,b)$ and $U := (a,a,a)$ on $X$. Then $R, S, T$ are trivially regular, while $U$ trivially fails to be regular. Let $F$ be a free Lie algebra with free generating set $x,x_1,x_2,x_3$ of $F$. Define $P := [x_1,[x_2,x_3]]$ and $Q := [[x_1,x_2],x_3]$. The Jacobi-identity then gives the linearisations:
\begin{eqnarray*}
\phantom{o} [x,P] & = & [x,x_1,x_2,x_3] - [x,x_1,x_3,x_2] - [x,x_2,x_3,x_1] + [x,x_3,x_2,x_1], \\
\phantom{o} [x,Q] & = & [x,x_1,x_2,x_3] - [x,x_2,x_1,x_3] - [x,x_3,x_1,x_2] + [x,x_3,x_2,x_1].
\end{eqnarray*}
So $\mathcal{S}_P := \{ \id, (23) ,(123) , (13) \}$ and $\mathcal{S}_Q := \{ \id, (12) , (132) , (13) \}$. We see that $S$ is Lie-regular because of the first identity, and that $T$ is Lie-regular because of the second identity. Also: either identity guarantees that $R$ is Lie-regular. Since $U$ is constant, every permutation will fix $U$. We conclude that $U$ is not Lie-regular. 
\end{example}

\subsection{Hyper-derivations and fullness}

\begin{definition}[Derivations and hyper-derivations] Consider a finite alphabet $(A,<,\omega,X)$. A \emph{derivation} of an element $S \in W(A)$ is a regular element $T\in W^2(A)$ such that the underlying sequence $\pi(T)$ of $T$ is a segment of $S$. We say that $T$ is a \emph{hyper-derivation} of $S$, iff in addition $|\sigma(T)| < |\sigma(S)|$. \end{definition}

In particular: if $T$ is a hyper-derivation of $S$, then $\kappa(T) <_{\operatorname{Lex}} \kappa(S)$.

\begin{lemma} \label{LemmaFullShift} Consider an alphabet $\mathcal{A}$ and its shift $\mathcal{B}$. Let $S$ be a finite sequence on $\mathcal{A}$, and let $T$ be a derivation of $S$. If $T$ is full w.r.t. $\mathcal{B}$, then $S$ is full w.r.t. $\mathcal{A}$. \end{lemma}

\begin{proof} Note that $\sigma(T)$ can be translated into $\sigma(S)$. Note also that a regular segment $U$ of $T$ (w.r.t. $\mathcal{B}$) is a derivation of $S$, with the same weight as its underlying sequence (w.r.t. $\mathcal{A}$). \end{proof}

Let us now construct these hyper-derivations.

\begin{lemma}[Existence of obvious derivations] \label{LemmaObviousDerivations} Consider an alphabet $(A,<,\omega,X)$ and an element $S := (a_1,\ldots,a_k) \in W(A)$. Let $a := \min \{ a_1 , \ldots, a_k \}$. Suppose that there exist natural numbers $1 \leq i_1 < j_1 < i_2 < j_2 < \cdots < i_s < j_s \leq k$ such that for all $t \in \{ 1,2,\ldots,s \}$: $a_{i_t} = a$ and $a_{j_t} > a$. Then $S$ has a derivation of length $s$. \end{lemma}

\begin{proof} This is a simple induction on $s \in \N$. \end{proof}

\begin{definition} We define the family of functions $\map{(f_n)_{n \in \N}}{\N \times \N}{\N}$ to be the unique solution of the recurrence relation $f_n(\alpha+1,\lambda) = \lambda \cdot (f_n(\alpha,\lambda) + n),$ with initial conditions $f_n(1,\lambda) := n$. We also define the map $\mapl{f}{\N \times \N}{\N}{(\alpha,\lambda)}{\lambda \cdot \alpha^\lambda}.$ \end{definition}

Although we will not be needing the exact solution of this recurrence, we can mention that the maps $f_n(\alpha,\lambda)$ are a polynomial in $\lambda$ of degree $\alpha-1$: $f_n(\alpha,\lambda) = n \cdot (\lambda^{\alpha-1} + \sum_{1 \leq j \leq \alpha-1} \lambda^j) .$ We note that for all $u,v,w \in \N$, $f_w(u,v) \leq w \cdot f (v,u)$. Also: the number of sequences of length at most $v$ on $u$ letters is at most $f(u,v)$.

\begin{lemma} Consider a finite alphabet $\mathcal{A} := (A,<,\omega,X)$ and sequence $S \in W(A)$. If $\sigma(S)$ can be translated into $X$, then $|\sigma(S)| \leq |X|$, and $\sigma(S)$ does not contain $\mathcal{O}_{|X|}(x,y)$ with $x \in G$ and $y \in X$. \end{lemma}

\begin{proof} $(1).$ This follows from the fact that translations of subsets in $G$ preserve cardinality. $(2).$ If $\mathcal{O}_{|X|}(x,y) \subseteq \sigma(S)$, then $\mathcal{O}_{|X|}(u+x,y) \subseteq u + \sigma(S)$. \end{proof}

\begin{lemma}[Existence of derivations] \label{LemmaDerivations} Consider a finite alphabet $(A,<,\omega,X)$ and an element $S := (a_1,\ldots,a_k) \in W(\omega^{-1}(X))$ such that $\sigma(S)$ can be translated into $X$. If there is an $l \in \N$ with $k \geq f_{|X|}(|\alpha(S)|,l)$, then $S$ has a derivation of length $l$. \end{lemma}

\begin{proof} We may suppose that $k = f_{|X|}(|\alpha(S)|,l)$. Let proceed by induction on $\alpha := |\alpha(S)|$. If $\alpha = 1$, then $a_1 = \cdots = a_k$, and $\mathcal{O}_{|X|}(0,a_1) \subseteq \sigma(S)$. This contradicts our assumption. Now suppose that $\alpha > 1$. Then there exist $S_1,\ldots,S_l, \in W(\alpha(S))$ of length $f_{|X|}(\alpha -1,l)$, and $T_1,\ldots,T_l \in W(\alpha(S))$ of length $|X|$, such that $$ S = \pi_1( (S_1,T_1,S_2,T_2,\ldots,S_l,T_l) ) .$$ 

If some $S_j$ does not use all letters of $\alpha(S)$, then we may use the induction hypothesis. Else, every $S_j$ uses every letter of $\alpha(S)$. Then $y := \min \alpha(S) $ occurs in every $S_j$. If any of the $\pi_1((S_j,T_j))$ is constant, then $\sigma(S)$ contains an orbit $\mathcal{O}_{|X|}(u,v)$, with $v \in X$. This contradicts our assumption on $\sigma(S)$. So we may apply \emph{Lemma \ref{LemmaObviousDerivations}} to conclude that $S$ has a derivation of length $l$.
\end{proof}

\begin{proposition}[Existence of hyper-derivations] \label{PropositionHyperDerivations} Consider a finite alphabet $(A,<,\omega,X)$ and an element $S := (a_1,\ldots,a_k) \in W(\omega^{-1}(X))$ such that $\sigma(S)$ can be translated into $X$. If $k \geq |X| \cdot f(|\omega^{-1}(X)|, |X| + l + 1)$, for some $l \in \N$, then $S$ a hyper-derivation of length at least $l$. \end{proposition}

\begin{proof} The map $f_{|X|}$ is monotone in its first argument, and $|\alpha(S)| \leq |\omega^{-1}(X)|$. The previous lemma then guarantees the existence of a derivation $T := (t_1,t_2,\ldots,t_{|X| + l},t_{|X| +l+1})$ of $S$ of length $|X| + l+1$. Note that the elements $t_1,t_2,\ldots$ are all regular of the same type. So we may suppose that they are all of the form $(y,\ldots)$, for some fixed $y \in \omega^{-1}(X)$. Define $U := (t_1,\ldots,t_{|X| + l}) \in W^2(\omega^{-1}(X))$. If $t_1 = t_2 = \cdots = t_{|X| + l}$, then $\sigma(T)$ contains an orbit of the form $\mathcal{O}_{|X|}(r,t_1)$, with $t_1 \in X$. Then also $\sigma(S)$ contains such an orbit, and this contradicts our assumption. So $U$ is a regular element of $W^2(\omega^{-1}(X))$. \newline

Note that $\sigma(U) \subseteq \sigma(T)$ and $|\sigma(T)| \leq |\sigma(S)|$. Suppose that $\sigma(U) = \sigma(T)$. If $u \in \sigma(U) = \sigma(T)$, then also $u + y \in \sigma(T) = \sigma(U)$. We conclude that $\sigma(U), \sigma(T)$ and $\sigma(S)$ contain an orbit of the form $\mathcal{O}(r,y)$. This contradicts our assumption on $\sigma(S)$. We conclude that $|\sigma(U)| < |\sigma(S)|$. \end{proof}

\subsection{A recursive upper bound}

\begin{definition}[Generalised Higman-map] \label{DefinitionGeneralisedHigmanMap} We define a map $\mapl{H}{\N \times \N \times \N}{\N}{(a,b,c)}{H_a(b,c)}$ by the boundary conditions $H_a(1,c) := 1$, $H_a(b,1) := a$ and, for $b,c >1$, by the recursion 
\begin{eqnarray*}
H_a(b,c) := &\max & \{ a \cdot f (c,a + H_{a}{(b-1,a \cdot f(c-1,H_{a}{(b,c-1)}))}+1) , \\
& &  H_a(b-1,c) , H_a(b,c-1)\}.
\end{eqnarray*}
By specialisation, we obtain the maps $\mapl{H}{\N \times \N}{\N}{(m,n)}{H_m(m,n)}$ and $\mapl{H}{\N}{\N}{n}{H_n(n,n) = H(n,n)}$. The latter is called the \emph{generalised Higman map}. \end{definition}

By construction: if $b' \leq b$ and $c' \leq c$, then $H_a(b',c') \leq H_a(b,c)$.

\begin{theorem}[Fullness, conditional on the content] \label{TheoremConditionalFullness} Consider a finite alphabet $\mathcal{A} := (A,<,\map{\omega}{A}{G},X)$ and $S \in W(A)$. If $ \| S \| \geq H_{|X|} ( \kappa(S)),$ then $S$ is full.\end{theorem}

\begin{proof} Let $S$ be given by $(a_1,\ldots,a_k)$, and let us prove the statement by induction on the content, $\kappa(S) \in (\N \times \N,<_{\operatorname{Lex}})$, of $S$. In order to simplify the formulae, we abbreviate $\sigma := |\sigma(S)|$ and $\alpha := |\alpha(S)|$. Suppose first that $\sigma = 1$.  Then $k \geq 1$, and $(a_1)$ is a regular segment of weight in $G \setminus X$, so that $S$ is full. Suppose next that $\alpha = 1$. Then $k \geq |X|$, and $\mathcal{O}_{|X|}(a_1,a_1) \subseteq \sigma(S)$. If $\omega(a_1) \in X$, then $\sigma(S)$ contains an orbit with increment in $X$, so that $S$ is full. If $\omega(a_1) \in G \setminus X$, then $(a_1)$ is a regular element of weight in $G \setminus X$, so that again $S$ is full. We may therefore assume that $\sigma,\alpha > 1$. \newline

We may assume even more: the letters $a_1,\ldots,a_k$ of $S$ are all in $\omega^{-1}(X)$, and $\sigma(S)$ can be translated into $X$, since otherwise $S$ is trivially full. We may then apply proposition \ref{PropositionHyperDerivations}, and obtain a hyper-derivation $U := (x_1,\ldots,x_l) \in W(W(\alpha(S)))$ of $S$ with length $ l \geq H_{|X|}{(\sigma-1,r)},$ where $r := |X| \cdot f(\alpha - 1,H_{|X|}(\sigma,\alpha-1))$. We may assume that none of the $x_j$ is full, since otherwise $S$ is again trivially full. \newline

In order to complete the induction, we need only find a subset $B$ of $W(A)$ of size $|B| \leq r$ such that $U \in W(B)$. Indeed: the finite alphabet $\mathcal{B} := (B,<',\map{\omega'}{B}{G},X)$ is a restriction of a shift of $\mathcal{A}$, so that lemma \ref{LemmaFullShift} applies. By definition, there is a $y \in \alpha(S)$ such that each $x_i$ is in $W(\{ y \}) \times W(\alpha(S) \setminus \{ y \}) \subseteq W(A)$. Let $B$ be the set of regular elements of $W(A)$ in $W(\{ y \}) \times W(\alpha(S) \setminus \{ y \})$ that are not full and that have type at most $(\sigma,\alpha)$. The elements $x_1,\ldots,x_l$ all belong to $B$, so that indeed $U \in W(B)$. \newline

Let $Y \in B$. Then it is of the form $Y = \pi((Y_H,Y_T))$, with $Y_H \in W(\{y\})$ and $Y_T \in W(\alpha(S) \setminus \{ y \})$. Since $Y$ is not full, neither is $Y_H$. Since $\alpha(Y_H) = 1$, we have $\| Y_H \| < H_{|X|}(\kappa(Y_H)) := |X|$. Similarly: since $Y$ is not full, neither is $Y_T$. Since $\kappa(Y_T) <_{\operatorname{Lex}} \kappa(Y)$, we may apply the induction hypothesis to obtain $\| Y_T \| < H_{|X|}(\kappa(Y_T))$. Since $\sigma(Y_T) \leq \sigma$ and $\alpha(Y_T) \leq \alpha-1$, we get $\| Y_T \| \leq H_{|X|}(\sigma,\alpha-1)$. We see that each $Y \in B$ draws from $|X|$-many heads and $ H_{|X|}(\sigma,\alpha-1)$-many tails, and conclude that $|B| \leq r$. This finishes the proof. \end{proof}

\begin{proof}(Of theorem \ref{TheoremCharacterisation}) If $|\sigma(S)| > |X|$, then $S$ is full for trivial reasons. Else, $|\sigma(S)| \leq |X|$ and $|\alpha(S)| \leq |\omega^{-1}(X)|$. Then $\kappa(S) \leq_{\operatorname{Lex}} (|X|,|\omega^{-1}(X)|)$, and $$ \| S \| \geq H ( |X| , |\omega^{-1}(X)|) := H_{|X|}(|X|,|\omega^{-1}(X)|) \geq H_{|X|}(\kappa(S)). $$ Theorem \ref{TheoremConditionalFullness} guarantees that $S$ is full. Lemma \ref{LemmaRegularityImpliesLieRegularity} then completes the proof. \end{proof}

\subsection{Walks in Cayley-graphs}

Our results can also be stated in the language of Cayley-graphs. Let us consider the Cayley-graph $\Gamma := \Gamma(G,X)$ of $(G,+)$ with respect to $X$. For every finite walk $W := (v_1,v_2,\ldots,v_{n+1}) $ in $ \Gamma$, we define its sequence of edges to be $$E(W) := (v_2 - v_1,v_3-v_2,\ldots,v_{n+1}-v_n) \in X^{n}.$$ 
We define the walk $W$ to be (Lie)-regular, iff $E(W)$ is (Lie)-regular. We obtain:

\begin{corollary} If $X$ is arithmetically-free, then every sufficiently long walk in $X \subseteq \Gamma(G,X)$ contains a (Lie)-regular walk from $a$ to $b$ with $b-a \not \in X.$ In particular: every sufficiently long walk in $\Z_p^\times \subseteq \Gamma(\Z_p,\Z_p^\times)$ contains a Lie-regular cycle. \end{corollary}

\paragraph{Thanks.} The author would like to thank his host, Efim Zel'manov, and Lance Small for their hospitality during the Erwin Schr\"odinger Research Programme (\emph{Representations and gradings of solvable algebras}: $2013-2015$, $J3371-N25$) at the University of California, San Diego. He would also like to express his gratitude to the Erwin Schr\"odinger International Institute for Mathematical Physics where preliminary work for the research was done (\emph{Lie algebras: deformations and representations}). Finally, he thanks the Geometric and Analytic Group Theory-group of the University of Vienna.


\begin{thebibliography}{99999}



\bibitem{BahturinZaicev} {\sc Bahturin, Y. A.; Zaicev, M. V.}: Identities of graded algebras. J. Algebra 205 (1998), no. 1, 1-12. 

\bibitem{BKK} {\sc Benkart, G.; Kostrikin, A. I.; Kuznetsov, M. I.}: Finite-dimensional simple Lie algebras with a nonsingular derivation. J. Algebra 171 (1995), no. 3, 894-916.

\bibitem{BergenGrzeszczuk} {\sc Bergen, J.; Grzeszczuk, P.}: Gradings, derivations, and automorphisms of nearly associative algebras. J. Algebra 179 (1996), no. 3, 732-750. 


\bibitem{BorelSerre} {\sc Borel, A.; Serre, J.-P.}: Sur certains sous-groupes des groupes de Lie compacts.: Comment. Math. Helv. 27, (1953). 128-139.


\bibitem{BurdeMoensPeriodic} {\sc Burde, D.: Moens, W. A.}: {Periodic derivations and prederivations of Lie algebras} J. Algebra 357 (2012), 208-221. 



\bibitem{Higman} {\sc Higman, G.}: {Groups and rings having automorphisms without non-trivial fixed elements} J. London Math. Soc. 32 (1957), 321-334. 


\bibitem{JacobsonWeaklyClosed} { \sc Jacobson, N.} Une g\'en\'eralisation du th\'eor\`eme d'Engel. C. R. Acad. Sci. Paris 234, (1952). 579Ð581.

\bibitem{Jacobson} {\sc Jacobson, N.}: A note on automorphisms and derivations of Lie algebras. 
Proc. Amer. Math. Soc. 6, (1955). 281-283. 


\bibitem{KhukhroSupport} {\sc Khukhro, E. I.} Finite groups of bounded rank with an almost regular automorphism of prime order. Sibirsk. Mat. Zh. 43 (2002), no. 5, 1182--1191; translation in Siberian Math. J. 43 (2002), no. 5, 955Ð962 


\bibitem{KhukhroMakarenkoShumyatsky} {\sc Khukhro, E. I.; Makarenko, N.; Shumyatsky, P.}: {Nilpotent ideals in graded Lie algebras and almost constant-free derivations} Comm. Algebra 36 (2008), no. 5, 1869-1882.

\bibitem{KostrikinKuznetsov} {\sc Kostrikin, A. I.; Kuznetsov, M. I.}: {Two remarks on Lie algebras with nondegenerate derivation.} Trudy Mat. Inst. Steklov. 208 (1995), Teor. Chisel, Algebra i Algebr. Geom., 186-192. 


\bibitem{KrekninKostrikin} {\sc Kreknin, V. A.; Kostrikin, A. I.}: Lie algebras with regular automorphisms. (Russian) 
Dokl. Akad. Nauk SSSR 149 1963 249-251. 




\bibitem{Mattarei} {\sc Mattarei, S.}: The orders of nonsingular derivations of modular Lie algebras. Israel J. Math. 132 (2002), 265Ð275. 

\bibitem{Mattarei2} {\sc Mattarei, S.} The orders of nonsingular derivations of Lie algebras of characteristic two. Israel J. Math. 160 (2007), 23Ð40.

\bibitem{Mattarei3} {\sc Mattarei, S. }: A sufficient condition for a number to be the order of a nonsingular derivation of a Lie algebra. Israel J. Math. 171 (2009), 1-14. 

\bibitem{MoensPreprint} {\sc Moens, W.A.}: Arithmetically-free group-gradings of Lie-algebras: I, \emph{Preprint}.



\bibitem{Shalev} {\sc Shalev, A.}: Automorphisms of finite groups of bounded rank. Israel J. Math. 82 (1993), no. 1-3, 395-404. 

\bibitem{ShalevCoclass} {\sc Shalev, A.}: The structure of finite p-groups: effective proof of the coclass conjectures. Invent. Math. 115 (1994), no. 2, 315-345. 

\bibitem{ShalevOrder} {\sc Shalev, A.}: The orders of nonsingular derivations. Group theory. J. Austral. Math. Soc. Ser. A 67 (1999), no. 2, 254-260.



\bibitem{Shumyatsky}{\sc Shumyatsky, P.}: Applications of Lie ring methods to group theory. Nonassociative algebra and its applications (S‹o Paulo, 1998), 373Ð395, Lecture Notes in Pure and Appl. Math., 211, Dekker, New York, 2000. 

\bibitem{ShumyatskyTamarozziWilson} {\sc Shumyatsky, P.; Tamarozzi, A.; Wilson, L.}: $\Z_n$-graded Lie rings. J. Algebra 283 (2005), no. 1, 149-160. 


\bibitem{Thompson} {Thompson, J.}: Finite groups with fixed-point-free automorphisms of prime order. 
Proc. Nat. Acad. Sci. U.S.A. 45 1959 578-581. 



\bibitem{Efim}{Zel'manov, E.I.}: Lie algebras and torsion groups with identity. \emph{Preprint}, https://arxiv.org/abs/1604.05678.

\end{thebibliography}
\end{document}